\documentclass[12pt]{amsart}

\usepackage{color, graphicx}
\usepackage[usenames,dvipsnames,svgnames,table]{xcolor}
\usepackage{tikz}
\usepackage{url}

\usepackage{amsmath}
\usepackage{amsfonts}
\usepackage{latexsym, amssymb}
\usepackage[T1]{fontenc}
\usepackage[english]{babel}
\providecommand{\U}[1]{\protect\rule{.1in}{.1in}}

\newtheorem{theorem}{Theorem}[section]

\newtheorem{conjecture}[theorem]{Conjecture}
\newtheorem{corollary}[theorem]{Corollary}

\newtheorem{definition}[theorem]{Definition}
\newtheorem{example}[theorem]{Example}

\newtheorem{lemma}[theorem]{Lemma}
\newtheorem{notation}[theorem]{Notation}

\newtheorem{proposition}[theorem]{Proposition}

\def\N{\mathbb{N}}

\newcommand{\vs}[1]{\langle #1 \rangle}

\newcommand{\arrow}[2]{\overset{#1}{\xrightarrow{\hspace*{#2}}}}

\DeclareMathOperator{\gaps}{gaps}
\DeclareMathOperator{\cogaps}{cogaps}
\DeclareMathOperator{\maxgen}{maxgen}
\DeclareMathOperator{\pred}{pred}

\DeclareMathOperator{\mg}{mg}
\DeclareMathOperator{\mc}{mc}

\title{On the Gotzmann threshold of monomials}

\thanks{The authors gratefully acknowledge support from INdAM for visits at the University of Reggio Calabria in July 2019, July 2022 and February 2023, and from the Mathematics laboratory LMPA of Universit\'e du Littoral C\^ote d'Opale for a visit in February 2024.}

\author{V. Bonanzinga and S. Eliahou}

\date{\today}

\begin{document}
\maketitle

\begin{abstract} Let $R_n=K[x_1,\dots,x_n]$ be the $n$-variable polynomial ring over a field $K$. Let $S_n$ denote the set of monomials in $R_n$. A monomial $u \in S_n$ is a \textit{Gotzmann monomial} if the Borel-stable monomial ideal $\langle u \rangle$ it generates in $R_n$ is a Gotzmann ideal. A longstanding open problem is to determine all Gotzmann monomials in $R_n$. Given $u_0 \in S_{n-1}$, its \textit{Gotzmann threshold} is the unique nonnegative integer $t_0=\tau_n(u_0)$ such that $u_0x_n^t$ is a Gotzmann monomial in $R_n$ if and only if $t \ge t_0$. Currently, the function $\tau_n$ is exactly known for $n \le 4$ only. We present here an efficient procedure to determine $\tau_n(u_0)$ for all $n$ and all $u_0 \in S_{n-1}$. As an application, in the critical case $u_0=x_2^d$, we determine $\tau_5(x_2^d)$ for all $d$ and we conjecture that  for $n \ge 6$, $\tau_n(x_2^d)$ is a polynomial in $d$ of degree $2^{n-2}$ and dominant term equal to that of the $(n-2)$-iterated binomial coefficient
$$
\binom {\binom {\binom d2}2}{\stackrel{\cdots}2}.
$$
\end{abstract}


\smallskip


\section{Introduction}\label{sec introduction}

Let $K$ be a field. Let $R_n=K[x_1,\dots,x_n]$ be the $n$-variable polynomial algebra over $K$ endowed with the standard grading $\deg(x_i)=1$ for all $i$. We denote by $S_n=\{x_1^{a_1}\cdots x_n^{a_n} \mid a_i \in \N \ \text{for all }i\}$ the set of monomials in $R_n$, and by $S_{n,d} \subset S_n$ the subset of monomials $u$ of degree $\deg(u)=\sum_i a_i=d$.

A monomial ideal $J \subseteq R_n$ is said to be \emph{Borel-stable} or \emph{strongly stable} if for any monomial $v \in J$ and any variable $x_j$ dividing $v$, one has $x_iv/x_j \in J$ for all $1 \le i \le j$. Given a monomial $u \in S_n$, let $\vs{u}$ denote the smallest Borel-stable monomial ideal in $R_n$ containing $u$, and let $B(u)$ denote the unique minimal system of monomial generators of $\vs{u}$. Then $B(u)$ is the smallest set of monomials containing $u$ and stable under the operations $v \mapsto vx_i/x_j$ whenever $x_j$ divides $v$ and $i \le j$. See e.g.~\cite{FMS}.

Recall that a homogeneous ideal $I \subseteq R_n$ is a \textit{Gotzmann ideal} if, from a certain degree on, its Hilbert function attains Macaulay's lower bound. See \emph{e.g.} \cite{B,He} for more details. Determining which homogeneous ideals are Gotzmann ideals is notoriously difficult. See e.g.~\cite{AAH, AHH, Ho,HM,I, MH, Mu1,Mu2,OOS,PS,S} for more on Gotzmann ideals. In this paper, we focus on this question for principal-Borel monomial ideals $\vs{u}$. Whence the following definition.

\begin{definition} Let $u \in S_n$. We say that $u$ is a \emph{Gotzmann monomial} if its associated Borel-stable monomial ideal $\vs{u}$ is a Gotzmann ideal. 
\end{definition}

Even for this restricted family of ideals $\vs{u}$, determining those which are Gotzmann is very difficult. So far, it has only been solved for $n \le 4$. 

\medskip

The following theorem of Gotzmann is classical~\cite{G}.

\begin{theorem}\label{thm Gotzmann} Let $u \in S_n$. Then there exists an integer $\tau_n(u) \ge 0$ such that $ux_n^s$ is a Gotzmann monomial in $R_n$ if and only if $s \ge \tau_n(u)$.
\end{theorem}

For instance, $u$ is a Gotzmann monomial in $R_n$ if and only if $\tau_n(u)=~0$. 

\begin{definition} We shall refer to the above integer $\tau_n(u)$ as the \emph{Gotzmann threshold of $u$ in $R_n$}.
\end{definition}

Determining $\tau_n(u)$ for all $u \in S_n$ reduces to determine $\tau_n(u_0)$ for all $u_0 \in S_{n-1}$. Indeed, decomposing $u=u_0x_n^t$ with $u_0 \in S_{n-1}$ and $t \ge 0$, Theorem~\ref{thm Gotzmann} implies
$$
\tau_n(u) = \max(\tau_n(u_0)-t, \, 0).
$$
That is, if $t \ge \tau_n(u_0)$, then $u_0x_n^t$ is a Gotzmann monomial in $R_n$ whence $\tau_n(u)=0$, whereas if $t < \tau_n(u_0)$, then $ux_n^s=u_0x_n^{t+s}$  is a Gotzmann monomial in $R_n$ if and only if  $t+s \ge \tau_n(u_0)$ if and only if $s \ge \tau_n(u_0)-t$.

\medskip

Currently, the function $\tau_n$ is only known for $n=2,3,4$. Indeed, we have $\tau_2(u)=0$ for all $u \in S_2$ since every monomial in $S_2$ is a Gotzmann monomial, as is easily seen. For $n=3$, it follows from results in~\cite{Mu2} that for $u_0=x_1^ax_2^b \in S_2$, we have
$$
\tau_3(x_1^ax_2^b)=\binom b2.
$$
Finally, the case $n=4$ was settled in~\cite{BE}, where for $u_0 = x_1^ax_2^bx_3^c \in S_3$, the precise formula for $\tau_4(u_0)$ turned out to be rather intricate:
\begin{equation}\label{eq tau4}
\tau_4(x_1^ax_2^bx_3^c)= \binom{\binom b2}2+\frac{b+4}3 \binom b2+(b+1)\binom{c+1}2+\binom{c+1}3-c.
\end{equation}
It is a general fact that the exponent $a$ of $x_1$ plays no role for the Gotzmann property~\cite{BE}. The case where $a=c=0$ is the most interesting, being responsible for the main terms of the above formula:
$$
\tau_4(x_2^b)= \binom{\binom b2}2+\frac{b+4}3 \binom b2.
$$

\medskip
It is an open problem to determine $\tau_n(u_0)$ for $n \ge 5$ and all $u_0 \in S_{n-1}$. In this paper, we develop a general method to compute $\tau_n(u_0)$ for any given $u_0 \in S_{n-1}$. Based on the above observations for $n=4$, it makes sense to focus on the case $u_0=x_2^d$ as a priority. (We switch to exponent $d$ for the \textit{degree} of the resulting polynomial formula.) As an application of our method, we obtain in Section~\ref{sec x2} the formula
\begin{equation}\label{eq formula}
\tau_5(x_2^d) = \binom{\binom{\binom{d}{2}}{2}+\binom{d+1}{3}+\binom{d}{2}}{2}-\binom{\binom{d}{2}}{3}+\binom{d+3}{4}-d.
\end{equation}
It remains an open problem to determine $\tau_n(x_2^d)$ for all $n \ge 6$. Based on the cases $3 \le n \le 5$, we conjecture that $\tau_n(x_2^d)$ is a polynomial in $d$ of degree $2^{n-2}$ and dominant term equal to that of the $(n-2)$-iterated binomial coefficient
$$
\binom {\binom {\binom d2}2}{\stackrel{\cdots}2}.
$$

\subsection{Contents} In Section~\ref{sec back}, we recall or introduce basic notions such as the maxgen monomial, gaps and cogaps, in terms of which Gotzmann monomials can be characterized. In Section~\ref{sec mg}, we develop tools to effectively exploit these characterizations. Our main result to compute the Gotzmann threshold $\tau_n(u)$ of any monomial $u \in S_n$ is proved in Section~\ref{sec main}. Finally, in Section~\ref{sec x2}, we apply our method to compute the Gotzmann threshold in the significant case $u=x_2^d$ in $S_n$ for $n=5$. We conclude with a conjecture pointing to the expected intricacy of the formula for $\tau_n(x_2^d)$ for $n \ge 6$.

\section{Background}\label{sec back}

In this section, we recall or introduce concepts with which Gotzmann monomials in $R_n$ can be characterized.

\subsection{The maxgen monomial}

\begin{notation} Let $u \in S_n \setminus \{1\}$. Let $k \le n$ be the largest index such that  $x_k$ divides $u$. We then set $\max(u)=k$ and $\lambda(u) = x_k$.
\end{notation}

\begin{definition} Let $B \subseteq S_{n,d}$ be a set of monomials of degree $d \ge 1$. The \emph{maxgen monomial} of $B$ is defined as
$$
\maxgen(B) = \prod_{u \in B} \lambda(u).
$$
For the empty set, we define $\maxgen(\emptyset) =1$.
\end{definition}
For instance, $S_{3,2}=\{x_1^2,x_1x_2,x_1x_3,x_2^2,x_2x_3,x_3^2\}$, hence $\maxgen(S_{3,2})=x_1x_2^2x_3^3$. Note that $\deg(\maxgen(B)) = |B|$ for all $B \subseteq S_{n,d}$.

\subsection{Gaps and cogaps}
\begin{notation}\label{nota gaps} Let $u \in S_{n,d}$. We denote by $B(u) \subseteq S_{n,d}$ the smallest set of monomials containing $u$ and stable under the operations $v \mapsto vx_i/x_j$ whenever $x_j$ divides $v$ and $i \le j$. 
\end{notation}

Thus, $B(u)$ is the minimal set of monomial generators of the principal-Borel ideal $\vs{u}$ spanned by $u$.

\smallskip
We order the set $S_{n,d}$ lexicographically, with $\max(S_{n,d})=x_1^d$ and $\min(S_{n,d})=x_n^d$. For $u,v \in S_n$, we shall write $u \le v$ \textit{exclusively when $\deg(u)=\deg(v)$} and $u$ is lexicographically smaller than or equal to $v$.

\begin{notation} Let $u \in S_{n,d} \setminus \{x_1^d\}$. We denote by $\pred_n(u)$, or simply $\pred(u)$ when the context is clear, the predecessor of $u$ in $S_{n,d}$, i.e. the smallest $v \in S_{n,d}$ such that $u < v$. For $\ell \ge 1$, we recursively denote by $\pred^\ell(u)$ the $\ell$th predecessor of $u$ in $S_{n,d}$, provided $\pred^{\ell-1}(u) \neq x_1^d$.
\end{notation}

\smallskip
Given $u \in S_{n,d}$, we will consider upward paths from $u$ to some larger $v \in S_{n,d}$. In that respect, an \emph{elementary step} is the shortest possible upward path from $u$, namely from $u$ to its predecessor $\pred_n(u)$ in $S_{n,d}$.

\smallskip
Recall~\cite{BE} that if $m=\max(u) \ge 2$ and $u=u_0 x_m^a$ with $\max(u_0) \le m-1$ and $a \ge 1$, then
\begin{equation}\label{eq pred}
\pred_n(u_0 x_m^a)=u_0x_{m-1}x_n^{a-1}.
\end{equation}

\begin{notation} For $u \le v \in S_{n,d}$, we denote
\begin{eqnarray*}
L_n(u) & = & \{z \in S_{n,d} \mid z \ge u\}, \\
L_n^*(v,u) & = & \{z \in S_{n,d} \mid v > z \ge u\} = L_n(u) \setminus L_n(v).
\end{eqnarray*}
\end{notation}

The sets $L_n(u)$ and $L_n^*(v,u)$ are called \textit{lexsegments} and \textit{lexintervals}, respectively. By a slight abuse of notation, we shall drop the index $n$ and write $L(u)$ for $L_n(u)$ and $L^*(v,u)$ for $L_n^*(v,u)$.

\begin{notation} For $u \le v \in S_{n,d}$, we denote
$$\mu_n(u,v) = \maxgen(L^*(v,u))$$
and call $\mu_n(u,v)$ the \emph{cost} of the upward path from $u$ to $v$ in $S_{n,d}$.
\end{notation}

Note that $\mu_n(u,u)=1$.
We shall use the \emph{arrow notation}
$$
u \arrow{\mu_n(u,v)}{1.6cm} v 
$$
to indicate the upward path from $u$ to $v$ in $S_{n,d}$ together with its cost $\mu_n(u,v) = \maxgen(L^*(v,u))$. The most basic step is from $u$ to $\pred(u)$. Thus $\mu_n(u,\pred(u))=\lambda(u)$, the last variable dividing $u$. In arrow notation, this is summarized as 
$$
u \arrow{\lambda(u)}{1cm} \pred(u).
$$
\begin{example} Let $u=x_2^2x_4x_5, v=x_2^2x_3x_4 \in S_5$. Then $v > u$ and the elementary steps of the upward path from $u$ to $v$ in $S_5$ are
$$
u=x_2^2x_4x_5 \arrow{x_5}{0.6cm} x_2^2x_4^2 \arrow{x_4}{0.6cm} x_2^2x_3x_5 \arrow{x_5}{0.6cm} x_2^2x_3x_4=v. 
$$
Thus $\mu_5(u,v)=x_4x_5^2$, i.e. $u \arrow{x_4x_5^2}{1cm} v$.
\end{example}

\begin{notation}\label{nota u tilde} Let $u \in S_{n,d}$. Noting that $B(u) \subseteq L(u)$, we define 
$$
\begin{aligned}
\gaps_n(u) & = L(u) \setminus B(u), \\
\cogaps_n(u) & = L^*(\pred^g(u),u), \textrm{ where } g = |\gaps_n(u)| \\
& = L^*(\tilde{u},u), \textrm{ where } \tilde{u} =  \pred^g(u).
\end{aligned}
$$
\end{notation}

\begin{notation}\label{nota mc}
Let $n \ge 3$. We denote by $\mg_n$ and $\mc_n$ the maps from $S_n$ to $S_n$ defined by
$$
\begin{aligned}
\mg_n(u) & = \maxgen(\gaps_n(u)) \\
\mc_n(u) & = \maxgen(\cogaps_n(u)) \\
& = \mu_n(u, \tilde{u})
\end{aligned}
$$
for all $u \in S_n$. 
\end{notation}

\subsection{Characterizations of Gotzmann monomials}

Here is a first characterization of Gotzmann monomials in $R_n$ in terms of 
$\mg_n(u)$ and $\mc_n(u)$. 

\begin{theorem}[\cite{B,BE}]\label{thm VB} Let $u \in S_n$. Then $u$ is a Gotzmann monomial in $R_n$ if and only if 
$$
\mg_n(u)=\mc_n(u).
$$
\end{theorem}

In this paper, we shall mostly use the following more flexible characterization, which does not require one to first determine $\tilde{u}$ and $\mc_n(u)$ in order to apply Theorem~\ref{thm VB}.

\begin{corollary}\label{equiv gotzmann} Let $u \in S_n$. Then $u$ is a Gotzmann monomial in $R_n$ if and only if there exists $v \ge u$ such that $\mg_n(u) =\mu_n(u,v)$.
\end{corollary}
\begin{proof} Assume first that $u \in S_{n,d}$ is a Gotzmann monomial in $R_n$. Then $\mg_n(u)=\mc_n(u)=\mu_n(u,\tilde{u})$, where $\tilde{u}$ is introduced in Notation~\ref{nota u tilde}. Thus $\mg_n(u)=\mu_n(u,v)$ with $v=\tilde{u}$.

Conversely, assume $\mg_n(u) =\mu_n(u,v)$ for some $v \ge u$ in $S_{n}$. Let $g = \deg(\mg_n(u))=\deg(\mu_n(u,v))$. Then $u$ has $g$ gaps and the upward path from $u$ to $v$ has length $g$. That is, $v$ is the $g$th predecessor of $u$ in $S_n$, i.e. $v=\tilde{u}$ by definition of $\tilde{u}$. Thus $\mg_n(u) =\mu_n(u,\tilde{u})=\mc_n(u)$, whence $u$ is a Gotzmann monomial in $R_n$ by Theorem~\ref{thm VB}.
\end{proof}

\subsection{The method}
Corollary~\ref{equiv gotzmann} yields the following method to determine $\tau_n(u_0)$ for $u_0 \in S_{n-1}$. Using the tools developed in Section~\ref{sec mg}, one first computes $\mg_n(u_0x_n^t)$. One then computes costs $\mu_n(u_0x_n^t,v)$ of upward paths in $S_n$ starting from $u_0x_n^t$, and stop if and when such a cost is found to coincide with $\mg_n(u_0x_n^t)$. This justifies the following definition.

\begin{definition}\label{def target} Given $u_0 \in S_{n-1}$ and $t \ge 0$, we will refer to the monomial $\mg_n(u_0x_n^t) \in S_n$ as the \emph{target}. 
\end{definition}

Now, while looking to realize the target $\mg_n(u_0x_n^t)$ as the cost of upward paths in $S_n$ starting from $u_0x_n^t$, we will look to first realize truncated versions of it by ignoring the last variables.

\begin{notation} Let $n \ge i \ge 1$ be positive integers. We denote by
$$
\pi_i \colon S_n \to S_i
$$
the truncation morphism defined by $\pi_i(x_1^{a_1}\cdots x_n^{a_n}) = x_1^{a_1}\cdots x_i^{a_i}$. 
\end{notation}
By a slight abuse of notation, we may and will safely ignore the index $n$ in the notation $\pi_i$.

\begin{definition}\label{def i-target} Given $u_0 \in S_{n-1}$, $t \ge 0$ and $i \le n$, we will refer to the monomial $\pi_{i}(\mg_n(u_0x_n^t)) \in S_i$ as the \emph{$i$-truncated target}, or more shortly as the \emph{$i$-target}. 
\end{definition}

The effectiveness of the method is illustrated in Section~\ref{sec x2} with the  determination of formula~\eqref{eq formula} for the Gotzmann threshold $\tau_5(x_2^d)$.

\section{Computing $\mg_n(u)$ and $\mu_n(u,v)$}\label{sec mg}
In this Section, we develop tools enabling us to compute $\mg_n(u)$ for $u \in S_n$ and costs $\mu_n(u,v)$ of upward paths in $S_n$ starting from $u$.

\subsection{The map $\sigma_n$} The following map is used everywhere in the sequel.
\begin{notation} We denote by $\sigma_n \colon S_n \to S_n$ the map defined, for all $\displaystyle u=\prod_{i=1}^n x_i^{a_i} \in S_n$, by
$$
\sigma_n(\prod_{i=1}^n x_i^{a_i}) = \prod_{i=1}^n x_i^{a_1+\dots+a_i} = x_1^{a_1}x_2^{a_1+a_2}\cdots x_n^{a_1+a_2+\cdots+a_n}.
$$
\end{notation}

For instance, $\sigma_2(x_2)=x_2$, $\sigma_5(x_2)=x_2x_3x_4x_5$ and $\sigma_4(x_2^2x_3^5)=x_2^2x_3^7x_4^7$. 

\medskip
A useful property of $\sigma_n$ is its compatibility with the product. 
\begin{lemma}\label{sigma is morphism} Let $u,v \in S_n$. Then
$
\sigma_n(uv)=\sigma_n(u)\sigma_n(v).
$
\end{lemma}

\begin{proof} Set $u=\prod_{i=1}^n x_i^{a_i}$, $v=\prod_{i=1}^n x_i^{b_i}$. Then
$$
\begin{aligned}
\sigma_n(uv) & = \sigma_n(\prod_{i=1}^n x_i^{a_i+b_i})
= \prod_{i=1}^n x_i^{(a_1+b_1)+\dots+(a_i+b_i)} \\
& = \prod_{i=1}^n x_i^{(a_1+ \dots +a_i)+(b_1+\dots+b_i)}
= \sigma_n(u)\sigma_n(v). \qedhere
\end{aligned}
$$
\end{proof}

\medskip
Here is a simple recursive formula.

\begin{lemma}\label{sigma n} Let $u_0 \in S_{n-1}$. For all $t \ge 0$, we have
$$
\sigma_n(u_0x_n^{t})=\sigma_{n-1}(u_0)\cdot x_n^{t+\deg(u_0)}.
$$
\end{lemma}
\begin{proof} Follows from the definition of $\sigma_n$, noting that if $u_0=x_1^{a_1}\cdots x_{n-1}^{a_{n-1}}$, then
$$
{\deg(u_0)+t} = a_1+\cdots+a_{n-1}+t.
$$
\end{proof}

\medskip
We will also need to apply $\sigma_n$ iteratively. Here is the result.

\begin{proposition} Let $u=\prod_{i=1}^n x_i^{a_i}$. Then for all $t \ge 1$, one has
\begin{equation}\label{sigma^t}
\sigma_n^t(u) = \prod_{i=1}^{n} x_i^{\sum_{j=1}^i a_j\binom{t-1+i-j}{t-1}}.
\end{equation}
\end{proposition}
\begin{proof} By induction on $t$, using the recurrence relations of the binomial coefficients. For $t = 1$ we have $\sum_{j=1}^i a_j\binom{i-j}{0}=a_1+\dots+a_i$, as desired.
Assume now the formula true for some $t \ge 1$. Let us prove it for $t+1$. Set $b_i=\sum_{j=1}^i a_j\binom{t-1+i-j}{t-1}$, so that $\sigma_n^t(u)=\prod_{i=1}^{n} x_i^{b_i}$ by assumption. Now $\sigma_n^{t+1}(u)=\sigma_n(\sigma_n^t(u))=\sigma_n(\prod_{i=1}^{n} x_i^{b_i})=\prod_{i=1}^{n} x_i^{\sum_{j=1}^i b_j}$. For all $i$, set $c_i=\sum_{j=1}^i b_j$. We will show that $c_i=\sum_{j=1}^i a_j\binom{t+i-j}{t}$, thereby concluding the proof. We have
$$
\begin{aligned}
c_i &= \sum_{j=1}^i\left(\sum_{r=1}^j a_r\binom{t-1+j-r}{t-1}\right) \\
&= \sum_{r=1}^i a_r\sum_{j=r}^i\binom{t-1+j-r}{t-1}
\end{aligned}
$$
Setting $k=j-r$, we obtain
$$
c_i = \sum_{r=1}^i a_r \sum_{k=0}^{i-r}\binom{t-1+k}{t-1}.
$$
Using the well-known binomial formula
$$
\sum_{k=0}^N \binom{d+k}{d} = \binom{d+1+N}{d+1},
$$
and setting $r=j$, we obtain
$$
c_i= \sum_{j=1}^i a_j\binom{t+i-j}{t},
$$
as desired.
\end{proof}

For convenience, here is formula~\eqref{sigma^t} in expanded form:
\begin{equation}
\sigma_n^t(x_1^{a_1}\cdots x_n^{a_n}) = x_1^{a_1}x_2^{a_1t+a_2}x_3^{a_1\binom{t+1}{2}+a_2t+a_3}x_4^{a_1\binom{t+2}{3}+a_2\binom{t+1}{2}+a_3t+a_4}\cdots
\end{equation}

\medskip
The following result was first stated as Corollary 6.9 in~\cite{BE}.

\begin{proposition}\label{cor 6.9} Let $u \in S_n$. Then
$$
\mg_n(ux_n) = \sigma_n(\mg_n(u)).
$$
\end{proposition}

\begin{corollary}\label{cor mg t} Let $u \in S_{n}$. For all $t \ge 0$, we have $$\mg_n(ux_n^t)=\sigma_n^t(\mg_n(u)).$$
\end{corollary}

\subsection{Comparing maxgen monomials in $S_n$ and $S_{n-1}$}\label{sec maxgen}
The following result is Theorem 5.19 in~\cite{BE}. See Notation~\ref{nota gaps} for the set $B(u)$.

\begin{theorem}\label{thm 5.19} Let $u = x_{i_1}\cdots x_{i_d}$ with $i_1 \le \dots \le i_d \le n$. Then
$$
mg_n(u) \ = \ \prod_{k=1}^{d-1}\left( \prod_{j=i_{k+1}+1}^n x_j^{\binom{d-k-2+j-i_{k+1}}{d-k-1}}\right)^{|B(x_{i_1}\cdots x_{i_k})|-1},
$$
where the internal product is set to 1 if $i_{k+1}=n$.
\end{theorem}

\begin{proposition}\label{prop pi n-1} Let $u_0 \in S_{n-1}$. Then
$$
\pi_{n-1}(\mg_n(u_0)) = \mg_{n-1}(u_0).
$$
\end{proposition}

\begin{proof} Write $u=x_{i_1}\cdots x_{i_d}$ with $i_1 \le \dots \le i_d \le n-1$. By Theorem~\ref{thm 5.19} above, we have
$$
\mg_n(u) \ = \ \prod_{k=1}^{d-1}\left( \prod_{j=i_{k+1}+1}^n x_j^{\binom{d-k-2+j-i_{k+1}}{d-k-1}}\right)^{|B(x_{i_1}\cdots x_{i_k})|-1}.
$$
Let us now isolate the variable $x_n$. This yields
$$
\mg_n(u) \ = \ \prod_{k=1}^{d-1}\left( \prod_{j=i_{k+1}+1}^{n-1} x_j^{\binom{d-k-2+j-i_{k+1}}{d-k-1}} \cdot x_n^{\binom{d-k-2+n-i_{k+1}}{d-k-1}}\right)^{|B(x_{i_1}\cdots x_{i_k})|-1}.
$$
Again by Theorem~\ref{thm 5.19}, the left part computes $\mg_{n-1}(u_0)$, i.e.
$$
\mg_{n-1}(u_0) \ = \ \prod_{k=1}^{d-1}\left( \prod_{j=i_{k+1}+1}^{n-1} x_j^{\binom{d-k-2+j-i_{k+1}}{d-k-1}}\right)^{|B(x_{i_1}\cdots x_{i_k})|-1}.
$$
Summarizing, we have
\begin{equation}\label{exponent of xn}
\mg_{n}(u) \ = \ \mg_{n-1}(u_0) \cdot x_n^{\sum_{k=1}^{d-1} \binom{d-k-2+n-i_{k+1}}{d-k-1} \cdot (|B(x_{i_1}\cdots x_{i_k})|-1)}.
\end{equation}
This concludes the proof of the proposition.
\end{proof}

\begin{corollary}\label{cor pi mgn} Let $u_0 \in S_{n-1}$. For all $t \in \N$, we have
$$
\pi_{n-1}(\mg_n(u_0x_n^t)) = \sigma_{n-1}^t(\mg_{n-1}(u_0)) = \mg_{n-1}(u_0x_{n-1}^t).
$$
\end{corollary}

\begin{proof} Follows from Proposition~\ref{prop pi n-1} and induction on $t$ using the formula $\mg_n(ux_n)=\sigma_n(\mg_n(u))$ for any $u \in S_n$.
\end{proof}

\subsection{The polynomial $f(t)$}\label{sec f(t)}
In this section, to a given monomial $u_0 \in S_{n-1}$ with $n \ge 3$, we attach a polynomial $f(t)$ which will be involved in our main result towards determining $\tau_n(u_0)$. 

\begin{notation}\label{nota w(t)} Let $u_0 \in S_{n-1}$. For $t \in \N$, we denote by $f(t) \in \N$ the degree in $x_n$ of the monomial $\mg_n(u_0x_n^t) \in S_n$, i.e.
$$
f(t) = \deg_{x_n}(\mg_n(u_0x_n^t)).
$$
\end{notation}
Let us denote 
\begin{equation}\label{eq w(t)}
w(t)=\pi_{n-1}(\mg_n(u_0x_n^t)) \in S_{n-1}.
\end{equation}
Recall from Definitions~\ref{def target} and~\ref{def i-target} that $\mg_n(u_0x_n^t)$ and its truncation $w(t) \in S_{n-1}$ are called the \emph{target} and the \emph{$(n-1)$-target}, respectively.

It follows from the above corollary that
$$
w(t)=\mg_{n-1}(u_0x_{n-1}^t),
$$
whence the decomposition
\begin{equation}\label{eq mg(n-1)}
\mg_n(u_0x_n^t) = \mg_{n-1}(u_0x_{n-1}^t) x_n^{f(t)} = w(t)x_n^{f(t)}.
\end{equation}

\medskip
The proof of Proposition~\ref{prop pi n-1} yields an explicit formula for $f(t)$. 

\begin{theorem}\label{thm f(t)} Let $u_0=x_{i_1}\dots x_{i_r} \in S_{n-1}$ with $i_1 \le \dots \le i_r \le n-1$ and $n \ge 3$. For $t \ge 0$, set as above $\mg_n(u_0x_n^t) = w(t)x_n^{f(t)}$. Then
$$
f(t) = \sum_{k=1}^{r-1} \binom{t+r-k-2+n-i_{k+1}}{n-1-i_{k+1}} \cdot (|B(x_{i_1}\cdots x_{i_k})|-1).
$$
\end{theorem}
\begin{proof} We have $\deg(u_0)=r$. Fixing $t$, let $u=u_0x_n^t$ and $d=\deg(u)=r+t$. Formula~\eqref{exponent of xn} and the binomial formula $\binom{a}{b}=\binom{a}{a-b}$ directly yield the stated result.
\end{proof}

\begin{corollary} Fixing $n \ge 3$ and $u_0\in S_{n-1}$ not divisible by $x_1$, the above corresponding function $f(t)$ is a polynomial in $t$ of degree at most $n-3$.
\end{corollary}
\begin{proof} Writing $u_0=x_{i_1}\cdots x_{i_r}$ with $2 \le i_1 \le \dots \le i_r \le n-1$, the above formula for $f(t)$ is clearly a polynomial in $t$ of degree at most $n-1-i_{k+1}$ for any $1 \le k \le r-1$, and hence of degree at most $n-3$ since $i_{k+1} \ge 2$ for all $k$ by hypothesis.
\end{proof}

\begin{example} The following formulas were shown in~\cite{BE}. For $n=3$, we have $\mg_3(x_2^bx_3^t)=x_3^{\binom{b}{2}}$ independently of $t$, so that $f(t)=\binom{b}{2}$ for all $t$ here. For $n=4$, we have $\mg_4(x_2^bx_3^cx_4^t)=x_3^{\binom{b}{2}}x_4^{f(t)}$ where 
$$
f(t) = \binom b2 t+(\frac{b+1}3+c)\binom b2+(b+1)\binom{c+1}2+\binom{c+1}3-c,
$$
a polynomial of degree $1$. Incidently, note that
$$
\mg_4(x_2^bx_3^cx_4^t) = \mg_{3}(x_2^bx_3^c)x_4^{f(t)}.
$$
\end{example}

\subsection{Multiplying by $x_n$}
Given $u_0 \in S_{n-1}$, one needs to repeatedly multiply $u_0$ by $x_n$ until obtaining an equality $\mg_n(u_0x_n^t)=\mu_n(u_0x_n^t,v)$ for some $v \ge u_0x_n^t$  characterizing the Gotzmann property by Corollary~\ref{equiv gotzmann}. In that respect, one needs to compare the costs $\mu_n(u,v)$ and $\mu_n(ux_n,vx_n)$.

\begin{proposition}\label{mu and sigma} Let $u < v \in S_{n,d}$. Then
$\mu_n(ux_n,vx_n) \,=\, \sigma_n(\mu_n(u,v))$.
\end{proposition}

\begin{proof} Let $\ell$ be the length of the path from $u$ to $v$, that is, 
$$
\ell = |L^*(v,u)| = \deg(\mu_n(u,v)).
$$
Then $v=\pred^{\ell}(u)$. We proceed by induction on $\ell$.

\medskip
\noindent
\textbf{Case $\ell=1$.} Then $\mu_n(u,v)=x_i$ for some $2 \le i \le n$, i.e. 
$$
u \arrow{x_i}{1cm} v 
$$
in arrow notation. Hence $\max(u)=x_i$, and so $u=u'x_i^a$ with $a \ge 1$ and $\max(u') < \max(u)$. Recall that 
$$\sigma_n(x_i)=x_ix_{i+1}\cdots x_n$$ by definition of $\sigma_n$. As $v = \pred(u)$, it follows from~\eqref{eq pred} that
$$
v=u'x_{i-1}x_n^{a-1}.
$$
In arrow notation:
$$
u'x_i^a \arrow{x_i}{1cm} u'x_{i-1}x_n^{a-1}. 
$$
We now compute $\mu_n(ux_n,vx_n)$, separately for $i=n$ and $i < n$.

If $i=n$, then $ux_n=u'x_n^{a+1}$ and $vx_n=u'x_{n-1}x_n^{a}$. Therefore $vx_n=\pred(ux_n)$. In arrow notation, we have
$$
ux_n=u'x_n^{a+1} \arrow{x_n}{1.5cm} u'x_{n-1}x_n^{a}=vx_n
$$
Thus 
$\mu_n(ux_n,vx_n)=\mu_n(u,v) = x_n = \sigma_n(x_n)$, as desired. 

If $i < n$, then $ux_n=u'x_i^ax_n$ and $vx_n=u'x_{i-1}x_n^{a}$. The path from the former to the latter is as follows:
$$
ux_n=u'x_i^ax_n \arrow{x_nx_{n-1}\cdots x_{i+1}}{2.4cm} u'x_i^ax_i \arrow{x_i}{0.7cm} u'x_{i-1}x_n^{a}=vx_n
$$
Thus, by arrow composition, we have
$\mu_n(ux_n,vx_n)= x_i\cdots x_n = \sigma_n(x_i)$, as desired. This concludes the case $\ell=1$. 

\medskip
\noindent
\textbf{Case $\ell \ge 2$.} We assume by induction hypothesis that the statement is true for $\ell-1$. Let $u'=\pred(u)$. Then $v > u' > u$ and $\mu_n(u,v)=\mu_n(u,u')\mu_n(u',v)$. Since $\deg(\mu_n(u,u'))=1$, then $\deg(\mu_n(u',v))=\ell-1$. It follows from the case $\ell=1$ and the induction hypothesis that
\begin{eqnarray*}
\mu_n(ux_n,u'x_n) & = & \sigma_n(\mu_n(u,u')), \\
\mu_n(u'x_n, vx_n) & = & \sigma_n(\mu_n(u', v)).
\end{eqnarray*}
By arrow composition and the compatibility of $\sigma_n$ with the product, we have
\begin{eqnarray*}
\mu_n(ux_n,vx_n) & = & \mu_n(ux_n,u'x_n) \mu_n(u'x_n, vx_n) \\
& = & \sigma_n(\mu_n(u,u'))\sigma_n(\mu_n(u', v)) \\
& = & \sigma_n(\mu_n(u,u')\mu_n(u', v)) \\
& = & \sigma_n(\mu_n(u, v)),
\end{eqnarray*}
as desired. This completes the proof of the proposition.
\end{proof}

\section{Main result}\label{sec main}

Given $u_0 \in S_{n-1}$ with $n \ge 3$, we attached to it a polynomial $f(t)$ in Section~\ref{sec f(t)}. In this section, we further attach to $u_0$ two more polynomials $h(t), k(t)$. With them, our main result is the following expression in Theorem~\ref{thm tau(u0)} for the Gotzmann threshold of $u_0$ in $R_n$:
$$\tau_n(u_0)=f(0)-h(0)-k(0).$$
Several preliminaries are needed for the construction of $h(t), k(t)$ in Section~\ref{sec h and k}. Indeed, in view of reaching the equality $\mg_n(u_0x_n^t)=\mu_n(u_0x_n^t,v)$ for some $v \ge u_0x_n^t$, characterizing the Gotzmann property, we need tools allowing to first reach its truncation in $S_{n-1}$ using $\pi_{n-1}$.

\subsection{Truncating the cost $\mu_n$}
\begin{proposition}\label{prop mu n-1} Let $u_0 < v_0$ in $S_{n-1,d}$. Then 
$$\pi_{n-1}(\mu_{n}(u_0,v_0))=\mu_{n-1}(u_0,v_0).$$ That is, $\mu_{n}(u_0,v_0)=\mu_{n-1}(u_0,v_0)x_n^s$ for some $s \ge 0$.
\end{proposition}
\begin{proof} In $S_{n-1,d}$, consider the full upward path from $u_0$ to $v_0$:
$$
u_0 < u_1 < u_2 < \dots < u_k = v_0,
$$
where $u_{i+1}=\pred_{n-1}(u_i)$ for all $i$. Given $i \le k-1$, we have $u_i=w_ix_r^a$ where $w_i \in S_{r-1}$ and $a \ge 1$. So the upward step from $u_i$ to $u_{i+1}$ in $S_{n-1}$ is given by
$$
u_i=w_ix_r^a \arrow{x_r}{1cm} u_{i+1}=w_ix_{r-1}x_{n-1}^{a-1}.
$$
That is, we have $\mu_{n-1}(u_i,u_{i+1})=x_r$. Now the predecessor of $u_i$ \textit{in $S_n$} is given by $\pred_n(w_ix_r^a)=w_ix_{r-1}x_{n}^{a-1}$. Thus, the upward path from $u_i$ to $u_{i+1}$ in $S_n$ is given by
$$
u_i=w_ix_r^a \arrow{x_r}{1cm} w_ix_{r-1}x_{n}^{a-1} \arrow{x_n^{a-1}}{1cm} u_{i+1}=w_ix_{r-1}x_{n-1}^{a-1}.
$$
That is, we have $\mu_{n}(u_i,u_{i+1})=x_rx_n^{a-1}=\mu_{n-1}(u_i,u_{i+1})x_n^{a-1}$. Since $\mu_n(u_0,v_0)=\prod_{i=0}^{k-1}\mu_n(u_i,u_{i+1})$, the proof is complete.
\end{proof}

\subsection{The monomial $z_n(t)$}\label{sec z(t)}
\begin{proposition}\label{prop w(t)} Let $u_0 \in S_{n-1}$. Assume $t \ge \tau_{n-1}(u_0)$. Let $w(t)=\pi_{n-1}(\mg_n(u_0x_n^t))$. Then there exists $v(t) \in S_{n-1}$ such that
$$
w(t) = \pi_{n-1}(\mu_{n}(u_0x_n^t,v(t))).
$$
\end{proposition}
\begin{proof} By hypothesis on $t$, we have that $u_0x_{n-1}^t$ is a Gotzmann monomial in $R_{n-1}$. Hence by Corollary~\ref{equiv gotzmann}, there exists $v(t) \in S_{n-1}$ such that
\begin{equation}\label{eq mg to mu}
\mg_{n-1}(u_0x_{n-1}^t)=\mu_{n-1}(u_0x_{n-1}^t,v(t)).
\end{equation}
Now
$$
\begin{aligned}
w(t) &= \pi_{n-1}(\mg_n(u_0x_n^t)) \\
&= \mg_{n-1}(u_0x_{n-1}^t) \textrm{ by Corollary~\ref{cor pi mgn}} \\ 
&= \mu_{n-1}(u_0x_{n-1}^t,v(t)) \textrm{ by~\eqref{eq mg to mu}}\\
&= \pi_{n-1}(\mu_{n}(u_0x_{n-1}^t,v(t))) \textrm{ by Proposition~\ref{prop mu n-1}} \\
&= \pi_{n-1}(\mu_{n}(u_0x_{n}^t,v(t))).
\end{aligned}
$$
The last line follows from $u_0x_n^t \arrow{x_n^t}{0.5cm} u_0x_{n-1}^t \arrow{\mu_{n}(u_0x_{n-1}^t,v(t))}{2.4cm} v(t)$.
\end{proof}

\begin{definition}\label{def z(t)} Let $u_0 \in S_{n-1}$ with $n \ge 3$. For $t \ge 0$, we denote by $z_n(t)$, if it exists, the smallest monomial in $S_n$ such that $u_0x_n^t \le z_n(t)$ and
$$
\pi_{n-1}(\mg_n(u_0x_n^t))=\pi_{n-1}(\mu_n(u_0x_n^t,z_n(t))).
$$
\end{definition}
With the above notation, it follows from Proposition~\ref{prop w(t)} that $z_n(t)$ exists whenever $t \ge \tau_{n-1}(u_0)$.

\begin{example}
Let $u_0=x_2^2 \in S_{n-1}$ with $n=4$. On the one hand, for $t \ge 0$, we have $\mg_4(x_2^2x_4^t)=x_3x_4^{t+1}$ as follows from Theorem~\ref{thm 5.19}. Hence $\pi_3(\mg_4(x_2^2x_4^t))=x_3$. For $t \ge 1$, we claim that $z_4(t) = x_2^3x_4^{t-1}$. Indeed, the upward path in $S_4$ from $x_2^2x_4^t$ starts like this:
$$
x_2^2x_4^t \arrow{x_4^t}{1.5cm} x_2^2x_3^t  \arrow{x_3}{0.8cm} x_2^3x_4^{t-1}.
$$
Thus $\mu_4(x_2^2x_4^t ,x_2^3x_4^{t-1})=x_3x_4^t$, so 
$$\pi_3(\mu_4(x_2^2x_4^t ,x_2^3x_4^{t-1}))=x_3=\pi_3(\mg_4(x_2^2x_4^t))$$
as desired. However, for $t=0$, we claim that $z_4(0)$ does not exist. Indeed, $\pi_3(\mg_4(x_2^2))=x_3x_4$, whereas the first step of the upward path in $S_4$ from $x_2^2$ is $x_2^2 \arrow{x_2}{0.5cm} x_1x_4$, whence $x_2$ divides $\pi_3(\mu_4(x_2^2,v))$ for all $v > x_2^2$ in $S_4$, implying
$$
x_3 = \pi_3(\mg_4(x_2^2)) \neq \pi_3(\mu_4(x_2^2,v)).
$$
\end{example}

\begin{lemma}\label{interval monomials} Let $u \in S_{n,d}$ such that $x_n$ divides $u$. For all $v \in S_{n,d}$ such that $v > u$, we have that $\mu_n(u,v)$ is the product of monomials of the form
$$
x_ix_{i+1}\cdots x_n
$$
for various indices $2 \le i \le n$.
\end{lemma}
\begin{proof} Each elementary step in the upward path from $u$ to $v$ is of the form
$$
v_0x_j^a \arrow{x_j}{1cm} v_1=v_0x_{j-1}x_n^{a-1}
$$
where $2 \le j \le n$, $v_0 \in S_{j-1}$ and $a \ge 1$. Therefore, for the next elementary step starting from $v_1=v_0x_{j-1}x_n^{a-1}$, there are only two possibilities, namely
$$
\begin{cases}
v_0x_{j-1}x_n^{a-1} \arrow{x_{j-1}}{1cm} v_2=v_0x_{j-2} & \textrm{ if } a = 1, \\
v_0x_{j-1}x_n^{a-1} \arrow{x_{n}}{1cm} v_2=v_0x_{j-1}x_{n-1}x_n^{a-2} & \textrm{ if } a \ge 2.
\end{cases}
$$
This shows that if the cost of the current elementary step is $x_j$, then the cost of the next elementary step is either $x_{j-1}$ or $x_n$.
\end{proof}

\begin{proposition}\label{prop key} Let $u,v \in S_n$ such that $x_n$ divides $u$ and $v > u$. Assume that $x_k$ divides $\mu_n(u,v)$ for some $k\le n-1$. Then there exists $v' \in S_n$ such that $u \le v' < v$ and $\pi_k(\mu_n(u,v)/x_k) = \pi_k(\mu_n(u,v'))$.
\end{proposition}

\begin{proof} By hypothesis, there is an upward path of the form
$$
u \arrow{}{0.7cm} u_1 \arrow{x_k}{0.5cm} u_2 \arrow{}{0.7cm} v.
$$
Then $\mu_n(u,v)=\mu_n(u,u_1)x_k\mu_n(u_2,v)$. We may assume that $u_2$ is maximal with the property that $x_k$ divides $\mu_n(u,u_2)$. It follows from Lemma~\ref{interval monomials} that $\mu_n(u_2,v)$ is only divisible by variables $x_j$ with $j > k$. Hence $\pi_k(\mu_n(u,v)/x_k)=\pi_k(\mu_n(u,u_1))$. So $v'=u_1$ has the required property.
\end{proof}

\begin{theorem}\label{thm main} Let $u_0 \in S_{n-1}$ with $n\ge 3$. Let $t \ge \tau_{n-1}(u_0)$. As in Definition~\ref{def z(t)},  let $z_n(t) \ge u_0x_n^t$ be the smallest monomial such that
\begin{equation}\label{eq mu}
\pi_{n-1}(\mg_n(u_0x_n^t)) = \pi_{n-1}(\mu_n(u_0x_n^t,z_{n}(t))).
\end{equation}
Then $z_n(t+1)=z_n(t)x_n$.
\end{theorem}
\begin{proof} Denote $w(t) = \pi_{n-1}(\mg_n(u_0x_n^t))$, the current $(n-1)$-target.

\smallskip
\noindent
\textbf{Case 1: $w(t)=1$.} That is, $\mg_n(u_0x_n^t)$ is a power of $x_n$. Then the monomial $z_n(t)=u_0x_n^t$ satisfies the required properties, since 
$$
w(t)=1=\mu_n(u_0x_n^t,u_0x_n^t) = \pi_{n-1}(\mu_n(u_0x_n^t,u_0x_n^t)).
$$
In this case, we clearly have $z_n(t+1)=z_n(t)x_n$.

\smallskip
\noindent
\textbf{Case 2: $\deg(w(t)) \ge 1$.} By Corollary~\ref{cor pi mgn}, we have
$$
w(t)=\pi_{n-1}(\mg_n(u_0x_n^t))=\mg_{n-1}(u_0x_{n-1}^t),
$$
whence $x_{n-1}$ divides $w(t)$. Since $t \ge \tau_{n-1}(u_0)$, there exists $v'(t) \in S_{n-1}$ such that
$$
w(t)=\mu_{n-1}(u_0x_{n-1}^t,v'(t)).
$$
That is, we have the following upward path in $S_{n-1}$:
$$
u_0x_{n-1}^t \arrow{w(t)}{1.2cm} v'(t).
$$
For the upward path from $u_0x_{n-1}^t$ to $v'(t)$ in $S_n$, Proposition~\ref{prop mu n-1} implies 
$$
\mu_{n}(u_0x_{n-1}^t,v'(t)) = w(t)x_n^s
$$
for some $s \in \N$. Starting from $u_0x_n^t$, we get
$$
u_0x_n^t \arrow{x_n^t}{1.2cm} u_0x_{n-1}^t \arrow{w(t)x_n^s}{1.2cm} v'(t)
$$
with cumulated cost $\mu_n(u_0x_n^t,v'(t))=w(t)x_n^{g(t)}$ where $g(t)=t+s$.
By Proposition~\ref{prop key} and its proof, since $x_{n-1}$ divides $w(t)$, the path
$$
u_0x_n^t \arrow{w(t)x_n^{t+s}}{2cm} v'(t)
$$
decomposes as
$$
u_0x_n^t \arrow{\frac{w(t)}{x_{n-1}}x_n^{g''(t)}}{1.8cm}  u''(t) \arrow{x_{n-1}}{1cm} u'(t) \arrow{x_n^{g'(t)}}{1.5cm} v'(t)
$$
where $g'(t)+g''(t)=g(t)$. It follows that $\max(u'')=x_{n-1}$ and that, in this path, the elementary step leading to $u''$ is of the form
$$
u'''(t) \arrow{x_n}{1cm} u''(t),
$$
i.e. with $\pred_n(u'''(t))=u''(t)$. In particular, $u'''(t)$ is divisible by $x_n$. Let us denote $$u'''(t)=v(t)x_n^{a_n}$$ with $v(t) \in S_{n-1}$ and $a_n \ge 1$. Stopping at $v(t)x_n^{a_n}$, we obtain the path
$$
u_0x_n^t \arrow{\frac{w(t)}{x_{n-1}}x_n^{l(t)}}{2cm} v(t)x_n^{a_n}
$$
with $l(t)=g''(t)-1$. Continuing from here, we have
$$
v(t)x_n^{a_n} \arrow{x_n^{a_n}}{2cm} v(t)x_{n-1}^{a_n}.
$$
The cumulated cost so far is the product $\frac{w(t)}{x_{n-1}}x_n^{l(t)+a_n}$ which is not yet on target. But with just one more step, namely
$$
v(t)x_{n-1}^{a_n} \arrow{x_{n-1}}{1.2cm} v(t)x_{n-2}x_n^{a_n-1},
$$
the cumulated cost reaches $w(t)x_n^{l(t)+a_n}$, on target for first time relatively to $\pi_{n-1}$. Hence $$z_n(t)=v(t)x_{n-2}x_n^{a_n-1}.$$
We now multiply the extremities of the initial path by $x_n$. By Proposition~\ref{mu and sigma}, the cost of the new path is given as follows:
$$
u_0x_n^{t+1} \arrow{\sigma_n(\frac{w(t)}{x_{n-1}}x_n^{l(t)})}{2.5cm} v(t)x_n^{a_n+1}.
$$
Continuing from here, we have
\begin{equation}\label{path2}
v(t)x_n^{a_n+1} \arrow{x_n^{a_n+1}}{1.8cm} v(t)x_{n-1}^{a_n+1},
\end{equation}
with cumulated cost so far given by 
\begin{equation}\label{cost so far}
\sigma_n\big(\frac{w(t)}{x_{n-1}}x_n^{l(t)}\big)x_n^{a_n+1}.
\end{equation}
We will now simplify~\eqref{cost so far} as
\begin{equation}\label{simplification}
\sigma_n\big(\frac{w(t)}{x_{n-1}}x_n^{l(t)}\big)x_n^{a_n+1} = \frac{w(t+1)}{x_{n-1}}\cdot x_n^*
\end{equation}
using the next two claims.

\medskip
\noindent
\textbf{Claim 1}. $w(t+1)=\sigma_{n-1}(w(t))$.

Indeed, we have
$$
\begin{aligned}
w(t+1) &= \pi_{n-1}(\mg_n(u_0x_n^{t+1})) \\
&= \pi_{n-1}(\sigma_n(\mg_n(u_0x_n^{t}))) \ \textrm{ by Prop.~\ref{cor 6.9}}\\
&= \pi_{n-1}(\sigma_n(w(t)x_n^{f(t)})) \ \textrm{ by \eqref{eq mg(n-1)}} \\
&= \pi_{n-1}(\sigma_{n-1}(w(t))x_n^{f(t)+\deg w(t)}) \ \textrm{ by Lemma.~\ref{sigma n}}\\
&= \sigma_{n-1}(w(t)).
\end{aligned}
$$

\medskip
\noindent
\textbf{Claim 2}. $\sigma_{n-1}(\frac{w(t)}{x_{n-1}})=\frac{w(t+1)}{x_{n-1}}$.

Indeed, $w(t)=\frac{w(t)}{x_{n-1}}\cdot x_{n-1}$. Applying $\sigma_{n-1}$ to both sides, using Lemma~\ref{sigma is morphism}, and the formula $\sigma_{n-1}({x_{n-1}})=x_{n-1}$, we get
$$
\begin{aligned}
\sigma_{n-1}(w(t)) &= \sigma_{n-1}(\frac{w(t)}{x_{n-1}})\cdot \sigma_{n-1}(x_{n-1}) \\
&= \sigma_{n-1}(\frac{w(t)}{x_{n-1}})\cdot x_{n-1}.
\end{aligned}
$$
Hence 
\begin{equation}\label{sigma n-1}
\begin{aligned}
\sigma_{n-1}\big(\frac{w(t)}{x_{n-1}}\big)&=\frac{\sigma_{n-1}(w(t))}{x_{n-1}} \\
&= \frac{w(t+1)}{x_{n-1}} \ \textrm{ by Claim 1}.
\end{aligned}
\end{equation}
This settles Claim 2.

\smallskip
We are now in a position to prove~\eqref{simplification}:

$$
\begin{aligned}
\sigma_n\big(\frac{w(t)}{x_{n-1}}x_n^{l(t)}\big)x_n^{a_n+1} 
&= \frac{\sigma_{n-1}(w(t))}{x_{n-1}}x_n^{\deg(w(t))+l(t)+a_n} \ \textrm{ by \eqref{sigma n-1}} \\
&= \frac{w(t+1)}{x_{n-1}}\cdot x_n^{*}  \ \textrm{ by Claim 2}.
\end{aligned}
$$
Note that the cost given by~\eqref{simplification} is not yet on target relative to $\pi_{n-1}$. 
Going back to the path with that cost, we had reached $v(t)x_{n-1}^{a_n+1}$ in ~\eqref{path2}. Starting from there, the next step upwards is
$$
v(t)x_{n-1}^{a_n+1} \arrow{x_{n-1}}{0.8cm} v(t)x_{n-2}x_n^{a_n}.
$$
Now, the cumulated cost of the extended path is given by
$$
\frac{w(t+1)}{x_{n-1}}\cdot x_n^*\cdot x_{n-1} = w(t+1)\cdot x_n^*.
$$
This is the first time we reach the desired $(n-1)$-target $w(t+1)$. Hence $z_n(t+1)$ is given by the extremity of that extended path, namely $v(t)x_{n-2}x_n^{a_n+1}$. Hence
$$
\begin{aligned}
z_n(t+1) &= v(t)x_{n-2}x_n^{a_n} \\
&= v(t)x_{n-2}x_n^{a_n-1}\cdot x_n \\
&= z_n(t) \cdot x_n,
\end{aligned}
$$
as desired.
\end{proof}

\subsection{Partial conversions}
Given $v \in S_{m-1}$, we need to know the cost of the path from $vx_m^k$ to $vx_{m-1}^{\ell}x_m^{k-\ell}$ for $\ell \le k$, where we only partially convert some of the $x_m$ into $x_{m-1}$. Here is the required formula. 
\begin{proposition}\label{partial} Let $v \in S_{m-1}$ with $2 \le m \le n$. For any $1 \le \ell \le k$, we have
\begin{multline*}
\mu_n(vx_m^k,vx_{m-1}^{\ell}x_m^{k-\ell}) = x_m^\ell x_{m+1}^{\binom{k}{2}-\binom{k-\ell}{2}}x_{m+2}^{\binom{k+1}{3}-\binom{k+1-\ell}{3}}\cdots \\\ x_{n}^{\binom{k+n-m-1}{n-m+1}-\binom{k+n-m-1-\ell}{n-m+1}}.
\end{multline*}
\end{proposition}

\begin{proof} By arrow composition, we have
$$
vx_m^k \arrow{}{1.5cm} vx_{m-1}^{\ell}x_m^{k-\ell} \arrow{}{1.5cm} vx_{m-1}^k.
$$
The respective costs of the long path and of the rightmost one are known. The cost of the leftmost one is given by their quotient, namely
$$
\mu_n(vx_m^k,vx_{m-1}x_m^{k-\ell}) = \mu_n(vx_m^k,vx_{m-1}^{k})/\mu_n(vx_{m-1}x_m^{k-\ell},vx_{m-1}^{k}).
$$
Applying the relevant formula for the numerator and the denominator yields the stated result.
\end{proof}

\begin{lemma}\label{power} Let $v_2 < v_1 \in S_{n,d}$. Then $\mu_n(v_2,v_1)=x_n^b$ for some $b \in \N$ if and only if $\deg_{x_n}(v_2) \ge b$.
\end{lemma}
\begin{proof} Let $a=\deg_{x_n}(v_2)$, so that $v_2=v_2'x_n^a$ with $v_2' \in S_{n-1}$. Then $\pred^i(v_2)=v_2'x_{n-1}^ix_n^{a-i}$ for all $0 \le i \le a$. In particular, for $i=a$, the monomial $\pred^a(v_2)$ has max index equal to $n-1$. Hence $\mu_n(v_2, \pred^{a+1}(v_2))=x_{n-1}x_n^a$. Therefore $v_1 < \pred^{a+1}(v_2)$, i.e. $v_1 = \pred^{j}(v_2)$ for some $j \le a$.
\end{proof}

\subsection{The polynomials $f(t), h(t), k(t)$}\label{sec h and k}

Let $u_0 \in S_{n-1}$ with $n \ge 3$. Having already attached to $u_0$ the polynomial $f(t)$ in Section~\ref{sec f(t)}, we will now attach to it two more polynomials $h(t), k(t)$ involved in our formula for $\tau_n(u_0)$ in Theorem~\ref{thm tau(u0)}.

\begin{notation} Let $u_0 \in S_{n-1}$. For $t \ge 0$, let us denote
$$
w(t) = \pi_{n-1}(\mg_n(u_0x_n^t)).
$$
Recall that we refer to $w(t) \in S_{n-1}$ as the \emph{$(n-1)$-target}.
\end{notation}
Hence
\begin{equation}\label{eq f(t)}
\mg_n(u_0x_n^t)) = w(t) x_n^{f(t)},
\end{equation}
where $f(t) = \deg_{x_n} \mg_n(u_0x_n^t)$ was shown to be a polynomial in $t$ in Section~\ref{sec f(t)}.

\smallskip
We have seen in Section~\ref{sec z(t)} that for $t \ge \tau_{n-1}(u_0)$, there exists a smallest monomial $z_n(t) \ge u_0x_n^t$ such that 
\begin{equation}\label{eq z(t)}
\pi_{n-1}(\mu_n(u_0x_n^t,z_n(t)))=w(t).
\end{equation}
\begin{definition}\label{def h,k} With $u_0 \in S_{n-1}$ and the above notation, we define $h(t), k(t)$ for $t \ge \tau_{n-1}(u_0)$ as follows:
\begin{equation}\label{eq k(t)}
\begin{aligned}
h(t) &= \deg_{x_n} \mu_n(u_0x_n^t,z_n(t)), \\
k(t) &= \deg_{x_n} z_n(t).
\end{aligned}
\end{equation}
\end{definition}

We will show that $h(t)$ and $k(t)$ are polynomials in $t$. Let us first compare $f(t)$ and $h(t)$.

\begin{proposition}\label{prop Delta f} Let $u_0 \in S_{n-1}$. Then for all $t \ge \tau_{n-1}(u_0)$, we have
$$
f(t+1)-f(t) = \deg(w(t)) = h(t+1)-h(t).
$$
\end{proposition}
\begin{proof}
First for $f(t)$. On the one hand, we have 
$$\mg_n(u_0x_n^{t+1})=w(t+1)x_n^{f(t+1)}.$$
On the other hand, 
$$
\begin{aligned}
\mg_n(u_0x_n^{t+1}) &=\sigma_n(\mg_n(u_0x_n^{t})) \\
&= \sigma_n(w(t)x_n^{f(t)}) \\
&= \sigma_{n-1}(w(t))x_n^{f(t)+\deg(w(t))}.
\end{aligned}
$$
Since $\sigma_{n-1}(w(t)) \in S_{n-1}$, it follows by comparing the exponents of $x_n$ in these expressions for $\mg_n(u_0x_n^{t+1})$ that $f(t+1)=f(t)+\deg(w(t))$, as claimed.

\smallskip
Now for $h(t)$. On the one hand, we have 
$$
\mu_n(u_0x_n^{t+1},z_n(t+1)) = w(t+1)x_n^{h(t+1)}.
$$
On the other hand, we proved above that $z_n(t+1)=z_n(t)x_n$. Hence
$$
\begin{aligned}
\mu_n(u_0x_n^{t+1},z_n(t+1)) &= \mu_n(u_0x_n^{t}x_n,z_n(t)x_n) \\
&= \sigma_n(\mu_n(u_0x_n^{t},z_n(t))) \\
&= \sigma_n(w(t)x_n^{h(t)}) \\
&= \sigma_{n-1}(w(t))x_n^{h(t)+\deg(w(t))}.
\end{aligned}
$$
Again, since $\sigma_{n-1}(w(t)) \in S_{n-1}$, and by comparing the exponent of $x_n$ on these expressions of $\mu_n(u_0x_n^{t+1},z_n(t+1))$, we find that $h(t+1)=h(t)+\deg(w(t))$, as claimed.
\end{proof}

\begin{corollary} The functions $h(t)$ and $k(t)$ associated to $u_0 \in S_{n-1}$ are polynomials in $t$ for $t \ge \tau_{n-1}(u_0)$.
\end{corollary}
\begin{proof} Since $f(t)$ is a polynomial in $t$ as shown in Section~\ref{sec f(t)}, the function $t \mapsto f(t+1)-f(t)$, i.e. the discrete derivative of $f(t)$, is also a polynomial in $t$, of degree one less than the degree of $f(t)$. Since the map $t \mapsto h(t+1)-h(t)=f(t+1)-f(t)$ is a polynomial in $t$, it follows that its discrete primitive $h(t)$ is itself a polynomial in $t$, of degree one more, i.e. of the same degree as $f(t)$. Finally, since $z_n(t+1)=z_n(t)x_n$ for $t$ large enough, and since $k(t)=\deg_{x_n} z(t)$, we have $k(t+1)-k(t)=1$ for $t$ large enough. Hence $k(t)$ is a polynomial in $t$ of degree $1$ with dominant coefficient $1$.
\end{proof}

\subsection{Determining $\tau_n(u_0)$}

Let $u_0 \in S_{n-1}$ with $n \ge 3$. We now use the above polynomials $f(t), h(t), k(t)$ attached to $u_0$ and~\eqref{eq f(t)},~\eqref{eq z(t)},~\eqref{eq k(t)} to determine when $u_0x_n^t$ is a Gotzmann monomial in $R_n$.

\begin{theorem}\label{thm k(t)} Let $u_0 \in S_{n-1}$ with $n \ge 3$. Let $t \in \N$.
Then $u_0x_n^t$ is a Gotzmann monomial in $R_n$ if and only if
$$
k(t) \ge f(t) - h(t).
$$
\end{theorem}

\begin{proof} Given $t \ge 0$, let  $W(t)=\mg_n(u_0x_n^t)$. Thus 
$$
W(t) = w(t) x_n^{f(t)}.
$$
As in \cite{BE}, denote $\tilde{u}(t)=\widetilde{u_0x_n^t} \in S_n$ the unique monomial such that $|L(\tilde{u}(t))|=|B(u_0x_n^t)|$, i.e. 
\begin{equation}\label{u tilda}
\tilde{u}(t) = \pred^{|\gaps(u_0x_n^t)|}(u_0x_n^t).
\end{equation}
 
\smallskip
Assume first that $u_0x_n^t$ is a Gotzmann monomial in $R_n$. Then we have 
$$
W(t)=\mg_n(u_0x_n^t) = \mg_n(L^*(\tilde{u}(t), u_0x_n^t)).
$$
By the minimality of $z_n(t)$, we have
$$
z_n(t) \le \tilde{u}(t).
$$
By arrow composition, the arrow
$$
u_0x_n^t \arrow{\mu_n(u_0x_n^t,\tilde{u}(t))}{3cm} \tilde{u}(t)
$$
decomposes as
$$
u_0x_n^t \arrow{\mu_n(u_0x_n^t,z_n(t))}{2.3cm} z_n(t) \arrow{\mu_n(z_n(t),\tilde{u}(t))}{2.3cm} \tilde{u}(t).
$$
Hence 
\begin{equation}\label{quotient}
\mu_n(z_n(t),\tilde{u}(t)) = \frac{\mu_n(u_0x_n^t,\tilde{u}(t))}{\mu_n(u_0x_n^t,z_n(t))}.
\end{equation}
Now, we have 
\begin{eqnarray*}
\mu_n(u_0x_n^t,\tilde{u}(t)) & = & w_0(t)x_n^{f(t)}, \\
\mu_n(u_0x_n^t,z_n(t)) & = & w_0(t)x_n^{h(t)}.
\end{eqnarray*}
Hence, by \eqref{quotient}, we get
$$\mu_n(z_n(t),\tilde{u}(t)) = x_n^{f(t)-h(t)}.$$
It follows from Lemma~\ref{power} that $\deg_{x_n}z_n(t) =k(t) \ge f(t)-h(t)$.

\medskip
Conversely, assume $k(t) \ge f(t)-h(t)$. Hence $z_n(t)$ is divisible by $x_n^{f(t)-h(t)}$. By Lemma~\ref{power} again, taking as many predecessors, we have
$$
z_n(t) \arrow{x_n^{f(t)-h(t)}}{2.3cm} v(t)
$$
for some monomial $v(t) \in S_n$. Thus we have
$$
u_0x_n^t \arrow{w(t)x_n^{h(t)}}{2.1cm} z_n(t) \arrow{x_n^{f(t)-h(t)}}{2.1cm} v(t),
$$
which by arrow composition yields
$$
u_0x_n^t \arrow{w(t)x_n^{f(t)}}{2.5cm} v(t).
$$
Let $g(t)=\deg(w(t)x_n^{f(t)})$. Hence $v(t)=\pred^{g(t)}(u_0x_n^t)$. Since $w(t)x_n^{f(t)}=W(t)=\mg_n(u_0x_n^t)$, it follows that
$$g(t)=|\gaps(u_0x_n^t)|.$$ 
Hence $v(t)=\tilde{u}(t)$ by definition of $\tilde{u}(t)$. Recall that $L^*(\tilde{u}(t),u_0x_n^t)=\cogaps(u_0x_n^t)$ by definition. Therefore
$$
\mc_n(u_0x_n^t)=\maxgen(L^*(\tilde{u}(t),u_0x_n^t)) = w(t)x_n^{f(t)}.
$$
We conclude that
$$
\mc_n(u_0x_n^t)=\mg_n(u_0x_n^t).
$$
Hence $u_0x_n^t$ is a Gotzmann monomial in $R_n$, as desired.
\end{proof}

We are now in a position to determine the Gotzmann threshold of $u_0$ in $R_{n}$.
\begin{theorem}\label{thm tau(u0)} Let $u_0 \in S_{n-1}$ with $n\ge 3$. Let $f(t),h(t),k(t)$ be the polynomials attached to $u_0$. Then $\tau_n(u_0)=f(0)-h(0)-k(0)$.
\end{theorem}

\begin{proof} Assume $t \ge \tau_{n-1}(u_0)$.

\smallskip
\noindent
\textbf{Step 1.} Then $z_n(t+1)=z_n(t)x_n$ by Proposition~\ref{prop w(t)}.

\smallskip
\noindent
\textbf{Step 2.} We have $k(t)=t+k(0)$ for some constant $k(0)$. Indeed, $k(t) = \deg_{x_n}z_n(t)$ by Definition~\ref{def h,k}. It follows from Step 1 that $k(t+1)=k(t)+1$ for $t \ge \tau_{n-1}(u_0)$. Hence, viewing $k(t)$ as a polynomial in $t$, it is of degree $1$ with dominant coefficient $1$. Thus $k(t)=t+k(0)$ for some constant $k(0)$.

\smallskip
\noindent
\textbf{Step 3.}  By Theorem~\ref{thm k(t)}, we have that $u_0x_n^t$ is a Gotzmann monomial in $R_n$ if and only if $k(t) \ge f(t)-h(t)$.

\smallskip
\noindent
\textbf{Step 4.} By Proposition~\ref{prop Delta f}, we have
$f(t+1)-f(t)=h(t+1)-h(t)=\deg(w(t))$.

\smallskip
\noindent
\textbf{Step 5.} By Step 4, it follows that both $f(t)$ and $h(t)$ are polynomials in $t$, since $\deg(w(t))$ is a polynomial in $t$ by definition of $w(t)$ and Corollary~\ref{cor pi mgn}. Hence the expressions $f(0),h(0)$ are well defined, and $f(t)-h(t)=f(0)-h(0)$ by Step 4 again.

\smallskip
\noindent
\textbf{Step 6.} By Step 3, $u_0x_n^t$ is a Gotzmann monomial in $R_n$ if and only if $k(t) \ge f(t)-h(t)$. By Step 5, this holds if and only if $k(t) \ge f(0)-h(0)$, and by Step 2, this holds if and only if $t+k(0) \ge f(0)-h(0)$. Hence $\tau_n(u_0) = f(0)-h(0)-k(0)$ by Theorem~\ref{thm Gotzmann}. This concludes the proof of the theorem.
\end{proof}

\subsection{An example} Let $n=5$ and $u_0=x_2^2x_4 \in S_4$. As a first illustration of the method, before its application to the more demanding computation of $\tau_5(x_2^d)$ in Section~\ref{sec x2}, let us show that $\tau_5(u_0)=6$. We do so by determining the attached polynomials $f(t), h(t), k(t)$ and then using the formula $\tau_5(u_0)=f(0)-h(0)-k(0)$.

\smallskip
\noindent
$\bullet$ Applying Theorem~\ref{thm 5.19} to $u_0=x_2^2x_4=x_{i_1}x_{i_2}\cdots x_{i_d}$ with $d=3$ and $i_1=i_2=2, i_3=4$, we get
$$
\mg_5(x_2^2x_4)=x_3x_4^2x_5^5.
$$
\smallskip
\noindent
$\bullet$ Corollary~\ref{cor mg t} then implies 
$$
\mg_5(x_2^2x_4x_5^t)=x_3x_4^{t+2}x_5^{\binom{t+1}2+2t+5}.
$$
Hence $f(t)=\binom{t+1}2+2t+5$.

\smallskip
\noindent
$\bullet$ We have $w(t)=\pi_4(\mg_5(x_2^2x_4x_5^t))=x_3x_4^{t+2}$. So, we now seek the smallest monomial $z_5(t) \ge x_2^2x_4x_5^t$ such that
\begin{equation}\label{eq target}
\pi_4(\mu_5(u_0x_5^t,z_5(t)))=x_3x_4^{t+2}.
\end{equation}
To do that, we begin the upward path starting from $u_0x_5^t=x_2^2x_4x_5^t$ and stop just before attaining the $4$-target $x_3x_4^{t+2}$ given by~\eqref{eq target}: 
$$
\begin{aligned}
& x_2^2x_4x_5^t \arrow{x_5^t}{0.5cm} x_2^2x_4^{t+1} \arrow{x_4}{0.5cm} x_2^2x_3x_5^{t} \arrow{x_5^t}{0.5cm} x_2^2x_3x_4^{t} \arrow{x_4}{0.5cm} x_2^2x_3^2x_5^{t-1} \arrow{x_5^{t-1}}{0.5cm} \\
& x_2^2x_3^2x_4^{t-1} \arrow{x_4}{0.5cm} x_2^2x_3^3x_5^{t-2} \arrow{x_5^{t-2}}{0.5cm} \dots \arrow{x_5}{0.5cm} x_2^2x_3^tx_4 \arrow{x_4}{0.5cm} x_2^2x_3^{t+1} \arrow{x_3}{0.5cm} \\
& x_2^3x_5^t \arrow{x_5^t}{0.5cm} x_2^3x_4^t.
\end{aligned}
$$
At this point, the cost of the upward path from $x_2^2x_4x_5^t$ to $x_2^3x_4^t$, i.e. the product of the costs of each arrow, is given by
$$
\mu_5(x_2^2x_4x_5^t,x_2^3x_4^t) = x_3x_4^{t+1}x_5^{t+t+(t-1)+\dots+1}.
$$
That is, mapping $x_5$ to $1$, we obtain
$$
\pi_4(\mu_5(x_2^2x_4x_5^t,x_2^3x_4^t)) = x_3x_4^{t+1},
$$
just below the $4$-target~\eqref{eq target}. Reaching that target for the first time is achieved with one more elementary step:
$$
x_2^3x_4^t \arrow{x_4}{0.5cm} x_2^3x_3x_5^{t-1}.
$$
Hence $x_2^3x_3x_5^{t-1}$ is the smallest monomial such that $x_2^3x_3x_5^{t-1} \ge x_2^2x_4x_5^t$ and
$$
\pi_4(\mu_5(u_0x_5^t,x_2^3x_3x_5^{t-1}))=x_3x_4^{t+2}
$$
as required by~\eqref{eq target}. We conclude that
$$
z_5(t) = x_2^3x_3x_5^{t-1}.
$$

\smallskip
\noindent
$\bullet$ We are now in a position to determine the polynomials $h(t), k(t)$:
$$
h(t) = \deg_{x_5}\mu_5(x_2^2x_4x_5^t,z_5(t))=\binom{t+3}2-3,
$$
and $k(t)=\deg_{x_5} z(t) = t-1$. Hence $\tau_5(x_2^2x_4)=f(0)-h(0)-k(0)=6$.

\section{Application: the case $u_0=x_2^d$}\label{sec x2}
We now apply the above theory to determine the Gotzmann threshold of the monomial $u_0 = x_2^d$ in $S_5$ as a function of $d$. While our method works for any monomial in $S_{n-1}$, this particular case already concentrates all the complexity of the task while avoiding formulas with too many parameters as the one in~\eqref{eq tau4} for $\tau_4(x_1^ax_2^bx_3^c)$.
\begin{proposition} For all $n \ge 3$, we have
$$
\mg_n(x_2^d)=x_3^{\binom{d}{2}}x_4^{\binom{d+1}{3}}x_5^{\binom{d+2}{4}}\cdots x_n^{\binom{d+n-3}{n-1}}.
$$
\end{proposition}
\begin{proof} Follows from Theorem~\ref{thm 5.19} and induction on $d$, using the following summation formula~\cite{Lo}:
$$
\sum_{k=1}^{d-1}k\binom{d+n-k-4}{n-3}=\binom{d+n-3}{n-1}.
$$
\end{proof}
Recall that the Gotzmann thresholds of $x_2^d$ in $S_3$ and $S_4$ are given by the formulas
\begin{align*}
\tau_3(x_2^d) &= \binom{d}{2}, \\
\tau_4(x_2^d) &= \binom{\binom{d}{2}}{2} + \frac{d+4}{3}\binom{d}{2},
\end{align*}
from~\cite{Mu2} and~\cite{BE}, respectively. Here is the main result of this section. 
\begin{theorem} For all $d \ge 2$, the Gotzmann threshold of $x_2^d$ in $R_5$ is given by
$$
\tau_5(x_2^d) = \binom{\binom{\binom{d}{2}}{2}+\binom{d+1}{3}+\binom{d}{2}}{2}-\binom{\binom{d}{2}}{3}+\binom{d+3}{4}-d.
$$
\end{theorem}

\begin{proof} Fix $d \ge 2$. By the previous proposition for $n=5$, we have
$$
\mg_5(x_2^d)=x_3^{f_3}x_4^{f_4}x_5^{f_5},
$$
where 
$$
\begin{aligned}
f_3 &=\binom{d}{2}, \\
f_4 &=\binom{d+1}{3}, \\
f_5 &=\binom{d+2}{4}.
\end{aligned}
$$
Now, given $t \ge 0$, we have
$$
\mg_5(x_2^dx_5^t)=\sigma_5^t(\mg_5(x_2^d))=x_3^{f_3(t)}x_4^{f_4(t)}x_5^{f_5(t)},
$$
where 
$$
\begin{aligned}
f_3(t) &= f_3=\binom{d}{2}, \\
f_4(t) &= f_3\hspace{0.4mm} t+f_4, \\
f_5(t) &=f_3\hspace{0.4mm}\binom{t+1}{2}+f_4 \hspace{0.4mm} t+f_5.
\end{aligned}
$$

As a first step, we seek the smallest $t$ so that a monomial $z_4(t)$ exists with the property $z_4(t) > x_2^dx_5^t$ and $\pi_3(\mu_5(x_2^dx_5^t,z_4(t)))=x_3^{f_3}$. The various steps of the required path are as follows:
$$
\begin{aligned}
& x_2^dx_5^t \arrow{x_5^t}{5.5cm} x_2^dx_4^t \\
& x_2^dx_4^t \arrow{x_4^tx_5^{\binom{t}{2}}}{5.5cm} x_2^dx_3^t \\
& x_2^dx_3^t \arrow{x_3^{f_3-1}x_4^{\binom{t}{2}-\binom{t-f_3+1}{2}}x_5^{\binom{t+1}{3}-\binom{t-f_3+2}{3}}}{5.5cm} x_2^{d+f_3-1}x_3^{t-f_3+1}.
\end{aligned}
$$
The cost of the last step is justified by Proposition~\ref{partial}. At this point, the $3$-target  $x_3^{f_3}$ is almost reached. One last step will allow to reach that target:
$$
x_2^{d+f_3-1}x_3^{t-f_3+1} \arrow{x_3}{2cm} x_2^{d+f_3}x_5^{t-f_3}.
$$
Hence $z_4(t)=x_2^{d+f_3}x_5^{t-f_3}$, since this is the very first time we reach the $3$-target $x_3^{f_3}$. The cost of the upward path from $x_2x_5^t$ to $z_4(t)$ is given by the product of the respective costs of each of the above steps, i.e.
$$
\begin{aligned}
\mu_5(x_2^dx_5^t,z_4(t)) &= x_3^{f_3}x_4^{t+\binom{t}{2}-\binom{t-f_3+1}{2}}x_5^{t+\binom{t}{2}+\binom{t+1}{3}-\binom{t-f_3+2}{3}} \\
&= x_3^{f_3}x_4^{\binom{t+1}{2}-\binom{t-f_3+1}{2}}x_5^{\binom{t+2}{3}-\binom{t-f_3+2}{3}}.
\end{aligned}
$$
Set $h_4(t)=\binom{t+1}{2}-\binom{t-f_3+1}{2}$. Thus
\begin{equation}\label{h4}
\mu_5(x_2^dx_5^t,z_4(t)) = x_3^{f_3}x_4^{h_4(t)}x_5^{\binom{t+2}{3}-\binom{t-f_3+2}{3}}.
\end{equation}
By minimality of $z_4(t)$, we have $h_4(t) \le f_4(t)$. We know that the difference $f_4(t)-h_4(t)$ is constant. Let us call this constant $\delta_4$, so that
$$
\delta_4=f_4(t)-h_4(t)=f_4(0)-h_4(0).
$$
Since $\displaystyle h_4(0)=\binom{1-f_3}{2}$, we have
$$
\delta_4=\binom{d+1}{3}+\binom{1-f_3}{2}.
$$
Using the formula $\binom{-a}{2}=\binom{a+1}{2}$, we get
$$
\delta_4=\binom{d+1}{3}+\binom{f_3}{2}.
$$
Using the formulas $\tau_3(x_2^d)=f_3=\binom{d}{2}$ and $\delta_4=\tau_4(x_2^d)-\tau_3(x_2^d)$, we recover the formula from~\cite{BE}, namely
$$
\tau_4(x_2^d) = \binom{d}{2}+\binom{d+1}{3}+\binom{\binom{d}{2}}{2}.
$$
Using $h_4(t)=f_4(t)-\delta_4$, formula~\eqref{h4} becomes
$$\mu_5(x_2^dx_5^t,z_4(t)) = x_3^{f_3}x_4^{f_4(t)-\delta_4}x_5^{\binom{t+2}{3}-\binom{t-f_3+2}{3}}.$$
Our next task is to find the smallest monomial $z_5(t)$ with the property
$$
\pi_4(\mu_5(x_2^dx_5^t,z_5(t))) = \pi_4(\mg_5(x_2^dx_5^t))=x_3^{f_3}x_4^{f_4(t)}=x_3^{f_3}x_4^{h_4(t)+\delta_4}.
$$
That is, we seek $z_5(t)$ so that $\pi_4(\mu_5(z_4(t),z_5(t)))=x_4^{\delta_4}$. To that end, let us continue the path from $z_4(t)$ on.
$$
\begin{aligned}
& z_4(t)=x_2^{d+f_3}x_5^{t-f_3} \arrow{x_5^{t-f_3}}{1.2cm} x_2^{d+f_3}x_4^{t-f_3} \\
& x_2^{d+f_3}x_4^{t-f_3} \arrow{x_4^{\delta_4-1}x_5^{\binom{t-f_3}{2}-\binom{t-f_3-\delta_4+1}{2}}}{4.5cm} x_2^{d+f_3}x_3^{\delta_4-1}x_4^{t-f_3-\delta_4+1}.
\end{aligned}
$$
The cost of the last step is again justified by Proposition~\ref{partial}. The $4$-target is almost reached. One last step will allow to reach that target:
$$
x_2^{d+f_3}x_3^{\delta_4-1}x_4^{t-f_3-\delta_4+1} \arrow{x_4}{1.2cm} x_2^{d+f_3}x_3^{\delta_4}x_4^{t-f_3-\delta_4}.
$$
Since $\pi_4(\mu_5(z_4(t),x_2^{d+f_3}x_3^{\delta_4}x_4^{t-f_3-\delta_4}))=x_4^{\delta_4}$, we conclude 
$$z_5(t)=x_2^{d+f_3}x_3^{\delta_4}x_4^{t-f_3-\delta_4}.$$
 The cumulative cost from $z_4(t)$ to $z_5(t)$ is
$$
\mu_5(z_4(t),z_5(t))=x_4^{\delta_4}x_5^{(t-f_3)+\binom{t-f_3}{2}-\binom{t-f_3-\delta_4+1}{2}}.
$$
Hence
$$
\mu_5(x_2^dx_5^t,z_5(t))=x_3^{f_3}x_4^{f_4(t)}x_5^{h_5(t)},
$$
where 
$$
\begin{aligned}
h_5(t)& = \binom{t+2}{3}-\binom{t+2-f_3}{3}+(t-f_3)+\binom{t-f_3}{2}-\binom{t-f_3-\delta_4+1}{2} \\
& = \binom{t+2}{3}-\binom{t+2-f_3}{3}+\binom{t+1-f_3}{2}-\binom{t-f_3-\delta_4+1}{2}.
\end{aligned}
$$
Again, we know that $f_5(t)-h_5(t)$ is a constant, let us call it $\delta_5$. Thus
$$
\delta_5=f_5(t)-h_5(t)=f_5(0)-h_5(0).
$$
We also know that $\delta_5=\tau_5(x_2^d)-\tau_4(x_2^d)$. Now
$$
\begin{aligned}
f_5(0) &= \binom{d+2}{4}, \\
h_5(0) &= -\binom{2-f_3}{3}+\binom{1-f_3}{2}-\binom{-f_3-\delta_4+1}{2}.
\end{aligned}
$$
Using the formulas $\binom{-a}{3}=-\binom{a+2}{3}$ and $\binom{-b}{2}=\binom{b+1}{2}$, we have
$$
\begin{aligned}
\delta_5 &= f_5(0)-h_5(0) \\
&= \binom{d+2}{4}+\binom{2-f_3}{3}-\binom{1-f_3}{2}+\binom{1-f_3-\delta_4}{2} \\ 
&= \binom{d+2}{4}-\binom{f_3}{3}-\binom{f_3}{2}+\binom{f_3+\delta_4}{2} \\
&= \binom{d+2}{4}-\binom{\binom{d}{2}}{3}-\binom{\binom{d}{2}}{2}+\binom{\binom{d}{2}+\binom{d+1}{3}+\binom{\binom{d}{2}}{2}}{2}.
\end{aligned}
$$
We are now in a position to get an explicit formula for $\tau_5(x_2^d)$. We have
$$
\begin{aligned}
\tau_5(x_2^d) &= \tau_4(x_2^d)+\delta_5 \\
&= \binom{d}{2}+\binom{d+1}{3}+\binom{d+2}{4}-\binom{\binom{d}{2}}{3}+\binom{\binom{d}{2}+\binom{d+1}{3}+\binom{\binom{d}{2}}{2}}{2} \\
&= \binom{\binom{\binom{d}{2}}{2}+\binom{d+1}{3}+\binom{d}{2}}{2}-\binom{\binom{d}{2}}{3}+\binom{d+3}{4}-d.
\end{aligned}
$$
\end{proof}


\subsection{A conjecture}
Recalling the formulas of $\tau_n(x_2^d)$ for $n=3,4,5$, namely
$$
\begin{aligned}
\tau_3(x_2^d) &= \binom{d}{2}, \\
\tau_4(x_2^d) & = \binom{\binom{d}{2}}{2}+\binom{d+1}{3}+\binom{d}{2}, \\
\tau_5(x_2^d) &= \binom{\binom{\binom{d}{2}}{2}+\binom{d+1}{3}+\binom{d}{2}}{2}-\binom{\binom{d}{2}}{3}+\binom{d+3}{4}-d,
\end{aligned}
$$
we are led to the following conjecture.

\begin{conjecture} For all $n \ge 3$, the Gotzmann threshold $\tau_n(x_2^d)$ is a polynomial of degree $2^{n-2}$ in $d$ with dominant term
$$
\frac{d^{2^{n-2}}}{2^{2^{n-2}-1}}=2(d/2)^{2^{n-2}},
$$
i.e. the dominant term of the $(n-2)$-iterated binomial coefficient
$$
\binom {\binom {\binom d2}2}{\stackrel{\cdots}2}.
$$
Equivalently,
$$
\lim_{d \to \infty} \frac{\tau_n(x_2^d)}{\binom{\tau_{n-1}(x_2^d)}2} = 1.
$$
\end{conjecture}
By the above formulas, the conjecture holds true for $n=3,4,5$.

\bigskip
\bigskip

\noindent
{\small
\textbf{Authors addresses}

\medskip

\noindent
$\bullet$ Vittoria {\sc Bonanzinga}\textsuperscript{a,b},

\noindent
\textsuperscript{a} Univ. Mediterranea di Reggio Calabria, 
DIIES\\
\textsuperscript{b} Via Graziella (Feo di Vito), 89100 Reggio Calabria, Italia \\
\email{vittoria.bonanzinga@unirc.it}

\medskip

\noindent
$\bullet$ Shalom {\sc Eliahou}\textsuperscript{a,b},

\noindent
\textsuperscript{a}Univ. Littoral C\^ote d'Opale, UR 2597 - LMPA - Laboratoire de Math\'ematiques Pures et Appliqu\'ees Joseph Liouville, F-62100 Calais, France\\
\textsuperscript{b}CNRS, FR2037, France\\
\email{eliahou@univ-littoral.fr}
}

\end{document}